\documentclass{IEEEtran}
\usepackage{cite}
\usepackage{amsmath,amssymb,amsfonts}
\usepackage{algorithmic}
\usepackage{graphicx}
\usepackage{textcomp}
\def\BibTeX{{\rm B\kern-.05em{\sc i\kern-.025em b}\kern-.08em
    T\kern-.1667em\lower.7ex\hbox{E}\kern-.125emX}}
    
\usepackage{mathtools}
\usepackage{epsfig} % for postscript graphics files
\usepackage[dvipsnames]{xcolor}

\usepackage{pgf,tikz,pgfplots}
\pgfplotsset{compat=1.15}
\usepackage{mathrsfs}
\usetikzlibrary{arrows}
\definecolor{darkred}{rgb}{0.8, 0.25, 0.33}
\usepackage{subcaption}

\newcommand{\norm}[1]{\left\lvert#1\right\rvert}

\newcommand{\normC}[1]{\left\lVert#1\right\rVert}
\newcommand{\sigmaU}[1]{\bar{\sigma}(#1)}

\newcommand{\dtau}{\mathrm{d}\tau}
\newcommand\eps{\varepsilon}

\newcommand\M{M}
\newcommand\N{N}
\DeclareMathOperator{\e}{e}%\newcommand{\e}{\mathbf{e}} % Exponential

\newtheorem{theorem}{Theorem}
\newtheorem{definition}{Definition}
\newtheorem{lemma}{Lemma}
\newtheorem{corol}{Corollary}
\newtheorem{remark}{Remark}

\newtheorem{prop}{Proposition}
\newtheorem{example}{Example}
\newenvironment{proof}{{\it Proof :~}}{\hfill$\square$\\}

\begin{document}
\title{Necessary and sufficient stability condition for time-delay systems arising from Legendre polynomial approximation}
%Stability test for time-delay systems arising from Legendre approximation
%Necessary and sufficient stability condition for time-delay systems arising from Legendre approximation
\author{Bajodek, M. and Gouaisbaut, F. and Seuret, A. 
\thanks{Authors are with LAAS-CNRS, Université de Toulouse, CNRS, UPS, Toulouse, France (e-mail: mbajodek, fgouaisbaut, aseuret @laas.fr).}}

\maketitle

\begin{abstract}
Recently, necessary conditions of stability for time-delay systems based on the handling of the Lyapunov-Krasovskii functional have been studied in the literature giving rise to a new paradigm. Interestingly, the necessary condition for stability developed by Gomez et al. has been proven to be sufficient. It is presented as a simple positivity test of a matrix issued from the Lyapunov matrix. The present paper proposes an extension of this result, where the uniform discretization of the state has been replaced by projections on the first Legendre polynomials. Like in Gomez et al., the stability is guaranteed regarding the sign of the eigenvalues of a matrix, whose size is given analytically from convergence arguments. Compared to them, by relying on the supergeometric convergence rate of the Legendre approximation, the required order to ensure stability can be remarkably reduced. Thanks to this significant modification, it is possible to find an outer estimate of the stability regions, which converges to the expected stability regions with respect to the number of projections, as illustrated in the example section.
\end{abstract}


\section{Introduction}

Delays appear unavoidably as soon as time processing or analog-to-digital converters interfere in the communication between interconnected dynamical systems. Numerous numerical methods have been deployed to consider this latency and to analyze the stability of time-delay systems~\cite{kolmanovskii1986stability}. First of all, the D-partition~\cite{McKay1970} issued from the modulus-argument calculation is simple to implement and indicates the exact stability properties. Furthermore, stability areas can be inferred using quasi-polynomials approaches and the set-up of Mikhaïov diagrams~\cite{Mondie2005}. Then, approximated models derived from pseudo-spectral techniques such as collocation~\cite{Breda2005}, or \textit{tau}~\cite{Ito1991} methods have also been prevalent. Besides stability sets obtained using bifurcation analysis, the root locus is outlined. In the Laplace domain, frequency-sweeping delay-dependent tests have also been developed to avoid case-by-case studies. The $\mathcal{H}_{\infty}$ analysis provides accurate stability results~\cite{knospe2006} and even results in the design of controllers~\cite{Fridman2005,Wu2010}. Lastly, in the time domain, it is well-known that the existence of Lyapunov-Krasovskii functionals leads to a necessary and sufficient condition of stability~\cite{Kharitonov2013} even though the sufficiency is usually not numerically tractable. The implementation has only been recently made feasible by discretizing the Lyapunov-Krasovskii functional~\cite{Egorov2017}. Henceforth, tractable necessary and sufficient conditions can be formulated as a positive definiteness test of a certain matrix. This method has been applied to various classes of delay systems with single~\cite{Mondie2013}, integral~\cite{Campos2018}, neutral~\cite{Mondie2016} or multiple~\cite{Zhabko2021,Mondie2018c} delay types. This paper focuses on this last feature.\\
%integral Mondie2018a % neutral,Mondie2018b

For linear finite-dimensional systems with state matrix~$A$, stability is equivalent to the positive definiteness of a symmetric matrix $P$ solution of the so-called Lyapunov equation $PA+A^\top P=-I$. Concurrently, for linear infinite-dimensional systems with operator $\mathcal{A}$, stability is equivalent  to find a positive hermitian operator $\mathcal{P}$ solution of the Lyapunov equation $\mathcal{P}\mathcal{A}+\mathcal{A}^\ast\mathcal{P}=-I$ (see~\cite{Datko1970}). Certified implementation techniques need to be developed~\cite{Mondie2013,Zhabko2015} to use such a theoretical necessary and sufficient condition. From one side, the necessity is directly obtained by the construction of an approximated Lyapunov-Krasovskii functional~\cite{Mondie2014}. On the other side, the sufficiency is obtained asymptotically for sufficiently large approximated orders~\cite{Mondie2017}. In practice, the approximation is realized by discretizing the Lyapunov matrix appearing in the operator $\mathcal{P}$. The interpolated functions are selected on each evenly-spaced subinterval as polynomials (see piece-wise linear or splines schema~\cite{Gu2013,Zhabko2015,Zhabko2019}) or exponential kernels (see~\cite{Mondie2013,Mondie2014,Mondie2017}). The latter technique makes it possible to elegantly end up with point-wise evaluations of the Lyapunov matrix~$U$. Then, a necessary and sufficient condition of stability is expressed as the positive definiteness of a matrix, approximating $\mathcal{P}$, of size $n^\ast$~\cite{Mondie2019}. The estimation of the order $n^\ast$ to assess stability has also been given in~\cite{Mondie2021}. Nevertheless, this estimated order seems extremely large, pessimistic, and limited by the discretization schema, which leads us to the following questions. Is it possible to extend the methodology to other approximation techniques and to other support basis? Can the numerical complexity of the numerical test be reduced? In that direction, we propose here another way to approximate the Lyapunov-Krasovskii functional following the idea of projection on a Legendre polynomial basis~\cite{Seuret2015Hierarchy}. The selection of Legendre polynomials is already meaningful insofar \textit{tau}-Legendre models are very efficient to perform convergent simulations~\cite{Mokhtary2011} or convergent stability estimates in the linear matrix inequality framework~\cite{Bajodek2023}. By taking the benefits of Legendre approximation, especially its supergeometric convergence rate, new necessary and sufficient criterion of stability is derived and the estimated order $n^\ast$ is notably reduced compared to~\cite{Mondie2021}.\\

The article is organized as follows. Section~\ref{sec0} presents the complete Lyapunov-Krasovskii functional and recalls the necessity and sufficiency of the converse Lyapunov theorem. Section~\ref{sec1} is dedicated to the supergeometric convergence occurring when performing Legendre approximations. Then, our novel necessary and sufficient numerical condition of stability is exposed in Section~\ref{sec2}. The last section deals with computational issues and performances evaluations of our stability test.\\

\emph{Notations:} Throughout the paper, $\mathbb{N}$ and $\mathbb{R}^{m\times p}$ and $\mathbb{S}^m$ denote the set of natural numbers, real matrices of size $m\times p$ and symmetric matrices of size $m$, respectively. For any square matrix $M\in\mathbb{R}^{m\times m}$, $M^\top$ denotes the transpose of $M$ and $\mathcal{H}(M)$ stands for $M+M^\top$. For any matrix $M\in\mathbb{S}^{m}$, $M\succ 0$ means that $M$ is positive definite (i.e. the eigenvalues of $M$ are strictly positive). Furthermore, for any matrix $M$ in $\mathbb{R}^{m\times p}$, the 2-norm of $M$ is $\norm{M}=\sqrt{\sigmaU{M^\top M}}$, where $\bar{\sigma}$ defines the maximal eigenvalue. The vector $u=\mathrm{vec}(M)$ in $\mathbb{R}^{mp\times 1}$ collocates the columns of $M$ and the inverse operation is denoted $\mathrm{vec}^{-1}$ and verifies $\mathrm{vec}^{-1}(\mathrm{vec}(M))=M$. Moreover, $I_m$ is the identity matrix of size $m$, $\delta_{jk}$ denotes the Kronecker delta, symbol $\otimes$ represents the Kronecker product, matrix $\begin{bsmallmatrix}M_{1}&M_{2}\\\ast&M_{3}\end{bsmallmatrix}$ stands for $\begin{bsmallmatrix}M_{1}&M_{2}\\M_2^\top&M_{3}\end{bsmallmatrix}$ and $\mathrm{diag}(d_1,\dots,d_n)$ is the diagonal matrix with diagonal coefficients $d_1,\dots,d_n$. We also declare functions $\e^r$ and $\lceil r \rceil$ as the exponential and ceiling part of the real number $r$, respectively.
For functions $f_1, f_2$ from $\mathbb{N}$ to $\mathbb{R}$, equivalence $f_1(n)\sim f_2(n)$ means that $\frac{f_1}{f_2}(n)$ is finite as $n$ tends to infinity.
The set of piece-wise continuous functions from~$[-h,0]$ to~$\mathbb{R}^m$ is denoted $\mathcal{C}_{pw}(-h,0;\mathbb{R}^m)$. For any function $\varphi$ in this set, the induced norm is $\normC{\varphi} = \underset{[-h,0]}{\mathrm{sup}}\norm{\varphi(\tau)}$.
Denote also $\mathcal{C}_{\infty}(-h,0;\mathbb{R}^m)$, the set of smooth functions from~$[-h,0]$ to~$\mathbb{R}^m$. Finally, the Shimanov notation $x_t:\left\{ \begin{array}{l} [-h,0] \to\mathbb{R}^m\\ \tau \mapsto x_t(\tau)=x(t+\tau)\end{array}\right.$ will be used all along the paper.

\section{Lyapunov necessary and sufficient stability condition for time-delay systems}\label{sec0}

\subsection{Time-delay system and Lyapunov-Krasovskii functional}

Consider a linear time invariant time-delay system given by
\begin{equation}\label{eq:tds}
    \dot{x}(t) = A x(t) + A_d x(t-h),\quad \forall t\geq 0,
\end{equation}
where $h>0$ is the delay and matrices $A,A_d$ in $\mathbb{R}^{m\times m}$ are constant and known. Such a system is initialized by $x_0=\varphi$ in $\mathcal{C}_{pw}(-h,0;\mathbb{R}^m)$ and, for any $t\geq 0$, $x_t$ in $\mathcal{C}_{pw}(-h,0;\mathbb{R}^m)$ denotes the state of~\eqref{eq:tds}. % (i.e. $x(\tau)=f(\tau)$ for any $\tau\in[-h,0]$) 
%The function $x_t$ is defined for any time $t\geq 0$ by $x_t:\left\{ \begin{aligned} [-h,0] &\to\mathbb{R}^m\\ \theta &\mapsto x_t(\tau)=x(t+\tau)\end{aligne}\right.$ and belongs to $\mathcal{C}_{pw}(-h,0;\mathbb{R}^m)$.

\begin{definition}
The trivial solution of system~\eqref{eq:tds} is said to be exponentially stable if there exist $\kappa\geq 1$ and  $\alpha>0$ such that, for all $t\geq 0$ and $x_0\in\mathcal{C}_{pw}(-h,0;\mathbb{R}^m)$, $\normC{x_t}\leq \kappa \e^{-\alpha t}\normC{x_0}$ holds.
\end{definition}
In order to study the stability of system~\eqref{eq:tds}, recall the Lyapunov-Krasovskii functional introduced in~\cite{Kharitonov2013}:
\begin{equation}\label{eq:V}
\begin{array}{lcl}
    V(\varphi) \!&=&\! \displaystyle \int_{-h}^0\! \int_{-h}^0 \!\begin{bsmallmatrix}\varphi(0)\\\!\varphi(\tau_1)\!\\\!\varphi(\tau_2)\!\end{bsmallmatrix}^{\!\top}
    \!\Pi(\tau_1,\tau_2)
    \begin{bsmallmatrix}\varphi(0)\\\!\varphi(\tau_1)\!\\\!\varphi(\tau_2)\!\end{bsmallmatrix} \dtau_1\dtau_2,
\end{array}
\end{equation}
for any $\varphi\in\mathcal{C}_{pw}(-h,0;\mathbb{R}^m)$, where matrix $\Pi$ is given by
\begin{equation}
    \Pi(\tau_1,\tau_2) \!=\!
    \begin{bmatrix}\frac{U(0)}{h^2}&\frac{1}{2h}U^\top\!(h+\tau_1)A_d&\frac{1}{2h}U^\top\!(h+\tau_2)A_d\\
    \ast & \frac{1}{2h}I_m & \frac{1}{2}A_d^{\top}U(\tau_1-\tau_2)A_d\\
    \ast & \ast & \frac{1}{2h}I_m \end{bmatrix}\!,
\end{equation}
and where the Lyapunov matrix $U$ in $\mathbb{R}^{m\times m}$ is given by $U = \mathrm{vec}^{-1}(\mathcal{U})$ where $\mathcal{U}=\mathrm{vec}(U)$ is given analytically by
\begin{equation}\label{eq:U}
\mathcal{U}(\tau) =
\left\{\begin{array}{lcl}
    \begin{bsmallmatrix}I_{m^2}&0\end{bsmallmatrix}\e^{\tau\M}\N^{-1}\begin{bsmallmatrix}-\mathrm{vec}(I_m)\\0\end{bsmallmatrix} & \text{if} & \tau \geq 0,\\
    \begin{bsmallmatrix}0&I_{m^2}\end{bsmallmatrix}\e^{(h+\tau)\M}\N^{-1}\begin{bsmallmatrix}-\mathrm{vec}(I_m)\\0\end{bsmallmatrix} & \text{if} & \tau < 0,
\end{array}\right.
\end{equation}
with
\begin{equation}\label{eq:MN}
\begin{aligned} 
    \M &= 
    \begin{bsmallmatrix}
        A^\top \otimes I_{m} & A_d^\top \otimes I_{m}\\
        - I_{m} \otimes A_d^\top & - I_{m} \otimes A^\top
    \end{bsmallmatrix},\\
    \N &= 
    \begin{bsmallmatrix}
        A^\top\otimes I_{m} + I_{m}\otimes A^\top & A_d^\top \otimes I_{m}\\
        I_{m^2} & 0
    \end{bsmallmatrix} \!+\!
    \begin{bsmallmatrix}
        I_{m} \otimes A_d^\top& 0\\
        0 & -I_{m^2}
    \end{bsmallmatrix} \e^{h\M}.
\end{aligned}
\end{equation}
%Vec(ABD) = (D^\top \otimes A) Vec(B)

The authors of~\cite{Mondie2017} showed that it is the unique functional that satisfies
\begin{equation}\label{eq:Vdot}
    \dot{V}(x_t) = -\norm{x(t-h)}^2,
\end{equation}
along the trajectories $x_t$ of system~\eqref{eq:tds}.\\

\begin{remark}\label{rem:Lyapcond}
The Lyapunov matrix $U$ ensuring~\eqref{eq:Vdot} is unique if and only if matrix~$\N$ is non singular. As explained in~\cite{Kharitonov2013}, such a limitation is the Lyapunov condition and excludes all systems with eigenvalues $s_1,s_2$ satisfying $\norm{s_1+s_2}= 0$.
\end{remark}

\subsection{Necessary and sufficient stability condition}

Under the Lyapunov condition, the authors of~\cite{Mondie2017,Mondie2021,Kharitonov2013,Zhabko2019} provide sufficient and necessary conditions for the exponential stability of system~\eqref{eq:tds}, which are recalled below.

\begin{lemma}\label{lem:CN}
If system~\eqref{eq:tds} is exponentially stable, then there exists $\eta>0$ such that, 
\begin{equation}
    V(\varphi)\! \geq\! \eta \!\left(\!\!\norm{\varphi(0)}^2 \!+\! \frac{1}{h}\!\int_{-h}^0\!\!\norm{\varphi(\tau)}^2\dtau\!\!\right)\!,\forall\varphi\in\mathcal{C}_{pw}(-h,0;\mathbb{R}^{m}),\label{eq:CN}
\end{equation}
where $V$ is the functional defined by~\eqref{eq:V} satisfying~\eqref{eq:Vdot}.
\end{lemma}
\begin{lemma}\label{lem:CS}
Assume that system~\eqref{eq:tds} has an eigenvalue with a strictly positive real part. Then,
\begin{equation}
    \exists\; \varphi\in\mathcal{S},\quad V(\varphi) \leq -\eta_0 = -\frac{\e^{-2rh}}{4r}\mathrm{cos}^2(b_0)<0,\label{eq:CS}
\end{equation}
where the functional~$V$ is defined by~\eqref{eq:V} and~\eqref{eq:Vdot} and where $\mathcal{S}$ stands for the compact set given by
\begin{equation}\label{eq:S}
    \mathcal{S} = \left\{ \varphi\in\mathcal{C}_{\infty}(-h,0;\mathbb{R}^{m});\;
    \begin{matrix}\norm{\varphi(0)} = 1\\ \normC{\varphi^{(k)}}\leq r^{k},\; \forall k\in\mathbb{N}\end{matrix}
    \right\},
\end{equation}
%and where
%\begin{equation}\label{eq:eta0}
%    \eta_0 = \frac{\e^{-2rh}}{4r}\mathrm{cos}^2(b_0) > 0,
%\end{equation}
with a system dependent parameter $r$ given by
\begin{equation}\label{eq:r}
    r = \norm{A}+\norm{A_d},
\end{equation}
and with scalar $b_0$ the unique root on $[0,\frac{\pi}{2}]$ of the function $g(b):= \sin^4(b)\big((hr)^2+b^2\big)-(hr)^2$.
\end{lemma}

\begin{proof}
The proofs of both lemmas are postponed to Appendices~\ref{app:CN} and~\ref{app:CS}.
\end{proof}

\section{Preliminaries on Legendre polynomials} \label{sec1}

\subsection{Legendre approximation}

Legendre polynomials considered on $[-h,0]$ are defined by
\begin{equation} \label{eq:lk} 
\begin{array}{c}
    \forall k\in\mathbb{N},\;
    l_k(\tau) = (-1)^k\overset{k}{\underset{j=0}{\sum}}(-1)^j(\begin{smallmatrix}k\\j\end{smallmatrix})(\begin{smallmatrix}k+j\\j\end{smallmatrix})\left(\frac{\tau+h}{h}\right)^j,
\end{array}
\end{equation}
where $(\begin{smallmatrix}k\\j\end{smallmatrix})$ stands for the binomial coefficient~\cite{Gautschi2006legendre}. These polynomials $\{l_k\}_{k\in\mathbb{N}}$ form an orthogonal sequence of functions, which spans the space of square-integrable functions~\cite{Gautschi2006legendre}. %$\mathcal{L}_2(-h,0;\mathbb{R})$. 

For the sake of simplicity, introduce matrix $\ell_n$ in $\mathbb{R}^{nm\times m}$ given by
\begin{equation}\label{eq:ln}
\ell_n(\theta)\!=\!\begin{bmatrix} l_0(\theta)I_m&l_1(\theta)I_m&\dots& l_{n-1}(\theta)I_m\end{bmatrix}^{\!\top}\!\!\!,\;\forall \theta\in[-h,0].
\end{equation}

For any function $\varphi$ in $\mathcal{C}_{pw}(-h,0;\mathbb{R}^m)$ and any approximation order $n\in\mathbb{N}^\ast$, let us decompose
\begin{equation}\label{eq:expansion}
\forall \tau\in [-h,0],\quad \varphi(\tau)= \underbrace{\ell_n^\top(\tau)\Phi_n}_{\varphi_n(\tau)} + \tilde{\varphi}_{n}(\tau),
\end{equation}
where $\varphi_n(\tau)=\ell_n^\top(\tau)\Phi_n$ is the polynomial approximation and $\tilde{\varphi}_{n}(\tau)=\varphi(\tau)-\varphi_n(\tau)$ is the residual error. The vector $\Phi_n$ represents the normalized $n$ first polynomial coefficients of the function $\varphi$ and is defined by
\begin{equation}\label{eq:phin}
\Phi_n = \underbrace{\mathrm{diag}\left(\frac{1}{h},\dots,\frac{2n\!-\!1}{h}\right)\otimes I_m}_{\mathbf{I}_n} \int_{-h}^0\!\! \ell_n(\tau)\varphi(\tau) \dtau \in \mathbb R^{nm},
\end{equation}
where $\mathbf{I}_n=\int_{-h}^0\ell_n(\tau)\ell_n^\top(\tau)\dtau$ is the diagonal Gram-Schmidt normalization matrix.

% Gram-Schmidt 
In the sequel, the objective is to prove that the Legendre approximation $\varphi_{n}$ converges uniformly towards $\varphi$ with respect to $\tau$ and also to quantify its convergence rate on $\mathcal{S}$.

\subsection{Convergence of the Legendre remainder}

In light of the polynomial approximation theory~\cite{Cheney1982}, it results in an important convergence lemma.

\begin{lemma}\label{lem:convf}
For any function $\varphi$ in $\mathcal{S}$, the approximation error~$\tilde{\varphi}_{n}$ in~\eqref{eq:expansion} verifies, for any $\eps>0$,
\begin{equation}\label{eq:convf}
    \normC{\tilde{\varphi}_{n}} \leq \eps, \quad \forall n\geq \mathcal{N}(\eps),
\end{equation}
where $\mathcal{N}(\eps)$ is given by
\begin{equation}\label{eq:neps}
    \mathcal{N}(\eps) = \mathrm{max}\left(4, \left\lceil \frac{3}{2} + \mu \e^{1+\mathcal{W}\left(-\frac{\log(\rho\eps)}{\mu\e}\right)}\right\rceil\right),
\end{equation}
where
\begin{equation}\label{eq:constants}
        \mu = \frac{hr}{2},\quad \rho = \sqrt{\frac{2\lceil\mu\rceil}{\pi^3}} \frac{1}{\mu^{2}} \left(\frac{\mu\e}{\lceil\mu\rceil+\frac{1}{2}}\right)^{\lceil\mu\rceil+\frac{1}{2}},
        %\rho = \sqrt{\frac{2\lceil\frac{hr}{2}\rceil}{\pi^3}} \left(\frac{hr}{2}\right)^{-2} \left(\frac{hr\e}{2\lceil\frac{hr}{2}\rceil+1}\right)^{\lceil\frac{hr}{2}\rceil+\frac{1}{2}},
        %\sqrt{\frac{2\lceil\frac{hr}{2}\rceil}{\pi^3}} \left(\frac{hr}{2}\right)^{\lceil\frac{hr}{2}\rceil-\frac{3}{2}} \left(\frac{\e}{\lceil\frac{hr}{2}\rceil+\frac{1}{2}}\right)^{\lceil\frac{hr}{2}\rceil+\frac{1}{2}},
\end{equation}
and where the Lambert function~\cite{Corless1996} $$\mathcal{W}:\left\{\begin{aligned}\mathbb{R}_+&\to\mathbb{R}_+,\\z&\mapsto \mathcal{W}(z)=y,\end{aligned}\right.$$ where $y$ is uniquely defined by the relation $y\e^{y}=z$.
\end{lemma}

\begin{proof}
According to~\cite[Th.~~2.5]{Wang2012}, an upper bound of the Legendre approximation error~$\normC{\tilde{\varphi}_n}=\underset{[-h,0]}{\mathrm{sup}}\norm{\tilde{\varphi}_{n}(\tau)}$ is given by
\begin{equation}
    \normC{\tilde{\varphi}_n} \leq \sqrt{\frac{\pi^3}{2\lceil\mu\rceil}}
    \frac{\mu^{n-\lceil\mu\rceil}}{(n-\frac{3}{2})\dots(1+\lceil\mu\rceil+\frac{1}{2})},\quad\forall n\geq 4,\label{eq:borne1}
\end{equation}
using Legendre polynomials properties and $n-1-\lceil\mu\rceil$ successive integrations by parts. Applying the logarithm to~\eqref{eq:borne1} leads to
\begin{equation}\label{eq:ineq0}
\begin{aligned}
    \log\left(\sqrt{\frac{2\lceil\mu\rceil}{\pi^3}}\normC{\tilde{\varphi}_{n}}\right) \leq&\; \left(n-\left\lceil\mu\right\rceil\right)\log\left(\mu\right)\\
    &\; - \sum_{1+\lceil\mu\rceil}^{n-2} \log\left(k+\frac{1}{2}\right).
\end{aligned}
\end{equation}
Since the $\log$ function is monotonically increasing, we obtain
\begin{equation*}
\begin{aligned}
    \sum_{k=1+\lceil\mu\rceil}^{n-2}\!\!\log\!\left(\!k+\frac{1}{2}\!\right) &\geq \!\int_{x=\lceil\mu\rceil}^{n-2}\!\! \log\!\left(\!x+\frac{1}{2}\!\right)\mathrm{d}x,\\
    &= \left[\left(x+\frac{1}{2}\right)\log\left(\frac{x+\frac{1}{2}}{\e}\right)\right]_{\lceil\mu\rceil}^{n-2},%\\
    %&= \left(n-\frac{3}{2}\right)\log\left(\frac{n-\frac{3}{2}}{\e}\right) - \left(\lceil\mu\rceil+\frac{1}{2}\right)\log\left(\frac{\lceil\mu\rceil+\frac{1}{2}}{\e}\right),
\end{aligned}
\end{equation*}
where $\e$ denotes the exponential of $1$. Reordering the terms and introducing $\rho$ in~\eqref{eq:constants}, inequality~\eqref{eq:ineq0} becomes
\begin{equation}\label{eq:ineq1}
     \log\left(\rho\normC{\tilde{\varphi}_{n}}\right) \leq - \left(n-\frac{3}{2}\right)\log\left(\frac{n-\frac{3}{2}}{\mu\e}\right).
\end{equation}
Denoting $y_n:=\log\left(\frac{n-\frac{3}{2}}{\mu\e}\right)$, we look for the orders $n$ such that the upper bound given by~\eqref{eq:ineq1} is bounded by $\log(\lambda\eps)$. Then, the following inequality need to be satisfied
\begin{equation}
    -\,y_n\e^{y_n} \leq \frac{\log(\rho\eps)}{\mu\e}.
\end{equation}
From Lambert function definition~\cite{Corless1996}, it boils down to
\begin{equation}
    y_n:=\log\left(\frac{n-\frac{3}{2}}{\mu\e}\right) \geq \mathcal{W}\left(-\frac{\log(\rho\eps)}{\mu\e}\right).
\end{equation}
Therefore, the orders for which the previous inequality holds satisfy
\begin{equation}
    n \geq \frac{3}{2} +  \mu\e^{1+\mathcal{W}\left(-\frac{\log(\rho\eps)}{\mu\e}\right)}.
\end{equation}
Together with the initial constraint $n\geq 4$ to employ~\eqref{eq:borne1}, the expression of $\mathcal{N}(\eps)$ is retrieved, which concludes the proof.
\end{proof}

This result allows us to estimate an order that ensures that $\normC{\tilde{\varphi}_n}$ is upper bounded by $\eps>0$, for any $\varphi$ in $\mathcal{S}$. The relation between $\eps$ and such a minimal order $\mathcal{N}(\eps)$ is depicted in Fig.~\ref{fig:conv}. As expected for smooth functions~\cite{boyd2001,Wang2012}, the uniform convergence of Legendre approximation is supergeometric which means that $\eps=\mathcal{N}^{-1}(n)\sim\e^{-n\log(n)}$ as emphasized in formula~\eqref{eq:borne1}.

\begin{figure}[!t]
    \centering
    \includegraphics[width=7cm]{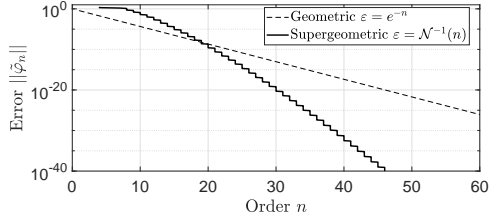}
    \caption{Convergence rate of $||\tilde{\varphi}_n||$ by Legendre approximation for functions $\varphi$ in $\mathcal{S}$ with $\mu=2$.}
    \label{fig:conv}
\end{figure}

\section{A new necessary and sufficient stability condition for time-delay systems}\label{sec2}

\subsection{Approximated Lyapunov-Krasovskii functional}

In this section, in order to construct an approximated Lyapunov-Krasovskii functional, the complete Lyapunov-Krasovskii functional given by~\eqref{eq:V} is regarded for particular functions $\varphi$, taken from subsets of $\mathcal{C}_{pw}(-h,0;\mathbb{R}^m)$. For instance, we consider here the space spanned by the $n$ first Legendre polynomial, and we take support on the $n$ first Legendre coefficients of $\varphi$ denoted $\Phi_n$ and expressed in~\eqref{eq:phin}. 
%For instance, discrete values of~$\varphi$ equally distributed on $[-h,0]$ have been considered in~\cite{Mondie2021,Mondie2019}. Here, as presented in the preliminaries section, we proceed differently by taking support on the $n$ first Legendre coefficients of $\varphi$ denoted $\Phi_n$ and expressed in~\eqref{eq:phin}. 

Let the approximated Lyapunov-Krasovskii functional at order $n$
\begin{equation}\label{eq:Vn}
    V_n(\varphi) = \begin{bsmallmatrix}\varphi(0)\\ \Phi_n \end{bsmallmatrix}^{\top}
    \mathbf{P}_n 
    \begin{bsmallmatrix}\varphi(0)\\ \Phi_n \end{bsmallmatrix},
\end{equation}
for any $\varphi\in\mathcal{C}_{pw}(-h,0;\mathbb{R}^m)$, with matrix
\begin{equation}\label{eq:Pn}
    \mathbf{P}_n = \begin{bmatrix}
        U(0) & \mathbf{Q}_n   \\
        \ast & \mathbf{T}_n + \mathbf{I}_n^{-1}
    \end{bmatrix}.
\end{equation}
In the previous expression, we have
\begin{equation}\label{eq:QT}
\begin{aligned}
    \mathbf{Q}_n &= \int_{-h}^0\! U^\top(h+\tau)A_d \ell_n^\top(\tau)\dtau,\\
    \mathbf{T}_n &= \int_{-h}^0\!\int_{-h}^0\!\! \ell_n(\tau_1)A_d^\top U(\tau_1-\tau_2)A_d \ell_n^\top(\tau_2)\dtau_1\dtau_2.
\end{aligned}
\end{equation}

\begin{remark}
Note that $V_n$ does not involve the Legendre remainder $\tilde{\varphi}_{n}$. Functional $V_n$ is an approximation of the Lyapunov-Krasovskii functional $V$ defined by~\eqref{eq:V}.
\end{remark}

Based on the previous section on polynomial approximation, the convergence of this approximated functional towards the complete Lyapunov-Krasovskii functional given by~\eqref{eq:V} will be established in the next section. 

\subsection{Convergence of the approximated Lyapunov-Krasovskii functional}

Define the Lyapunov-Krasovskii functional remainder as
\begin{equation}
    \tilde{V}_n(\varphi) = V(\varphi) - V_n(\varphi),\quad \forall \varphi\in\mathcal{C}_{pw}(-h,0;\mathbb{R}^m).
\end{equation}
Applying expansion~\eqref{eq:expansion}, this remainder is rewritten as
\begin{equation}\label{eq:Vntilde}
\begin{array}{lcl}
    \tilde{V}_n(\varphi) \!&=&\!\! \displaystyle \int_{-h}^0\! \int_{-h}^0 \!\begin{bsmallmatrix}\varphi(0)\\\!\varphi_{n}(\tau_1)\!\\\!\tilde{\varphi}_{n}(\tau_1)\!\\\!\tilde{\varphi}_{n}(\tau_2)\!\end{bsmallmatrix}^{\!\top}
    \!\Delta_n(\tau_1,\tau_2)\tilde{\varphi}_{n}(\tau_2)\dtau_1\dtau_2,
\end{array}
\end{equation}
where
\begin{equation}
    \Delta_n(\tau_1,\tau_2) \!=\!
    \begin{bmatrix}\frac{2}{h}U^\top(h+\tau_2)A_d\\
    2A_d^\top U(\tau_1-\tau_2)A_d\\
    A_d^{\top}U(\tau_1-\tau_2)A_d\\
    \frac{1}{h}I_m \end{bmatrix}.
\end{equation}

The main idea is now to prove, at least in the compact subset~$\mathcal{S}$ of $\mathcal{C}_{pw}(-h,0;\mathbb{R}^m)$ given by~\eqref{eq:S}, that the approximated Lyapunov-Krasovskii functional~$V_n$ given by~\eqref{eq:Vn} converges towards the complete Lyapunov-Krasovskii functional~$V$ given by~\eqref{eq:V} with a guaranteed and quantified convergence rate. %On $\mathcal{S}$, it boils down to the study of the convergence of~$\tilde{V}_n$ towards zero, presented in the next lemma.

\begin{lemma}\label{lem:conv}
For any $\varphi$ in $\mathcal{S}$ and $\eta>0$, we have
\begin{equation}
    \norm{\tilde{V}_n(\varphi)} \leq \eta, \quad \forall n\geq \mathcal{N}(\mathcal{E}(\eta)),
\end{equation}
where the Lyapunov-Krasovskii remainder $\tilde{V}_n$ is given by~\eqref{eq:Vntilde}, where the order $\mathcal{N}(\eps)$ is described in~\eqref{eq:neps} and where
\begin{equation}\label{eq:epsast}
%\begin{aligned}
    \mathcal{E}(\eta) = -\frac{\kappa_1+\kappa_2}{\kappa_2+1}+\sqrt{\left(\frac{\kappa_1+\kappa_2}{\kappa_2+1}\right)^2+\frac{\eta}{h(\kappa_2+1)}},
    %&= \frac{\eta}{(\kappa_1+\kappa_2)+\sqrt{(\kappa_1+\kappa_2)^2+(h+\kappa_2)\eta}},
%\end{aligned}
\end{equation}
with scalars $\kappa_1$ and $\kappa_2$ given by
\begin{equation}
    \kappa_1 = \underset{[0,h]}{\mathrm{max}}\norm{U(\tau)A_d},\quad \kappa_2 = \underset{[-h,h]}{\mathrm{max}}\norm{A_d^\top U(\tau)A_d}.
\end{equation}
\end{lemma}
\begin{proof}
An upper bound of $|\tilde{V}_n|$ is obtained by
\begin{equation}
\begin{array}{rcl}
    \norm{\tilde{V}_n(\varphi)} \!\!\!\!&\leq&\!\!\!\!\displaystyle 2 \int_{-h}^0\!\!  \left(\kappa_1\norm{\varphi(0)}+\kappa_2\norm{\varphi_n(\tau)}\right)\norm{\tilde{\varphi}_n(\tau)}\dtau\\
    &&\!\!\!\!\displaystyle + (\kappa_2+1)\int_{-h}^0 \norm{\tilde{\varphi}_n(\tau)}^2\dtau.
\end{array}
\end{equation}
Hence, having $\varphi$ in $\mathcal{S}$, $\normC{\varphi}=\underset{[-h,0]}{\mathrm{sup}}\norm{\varphi(\tau)}\leq 1$ so that
\begin{equation}
    \frac{1}{h}\norm{\tilde{V}_n(\varphi)} \leq 2(\kappa_1+\kappa_2) \normC{\tilde{\varphi}_{n}} + (\kappa_2+1)\normC{\tilde{\varphi}_{n}}^2.
\end{equation}
We obtain $\norm{\tilde{V}_n(\varphi)}\leq \eta$ under the following quadratic constraint
\begin{equation}
   -\frac{\eta}{h(\kappa_2+1)} + 2 \frac{\kappa_1+\kappa_2}{\kappa_2+1} \normC{\tilde{\varphi}_{n}} + \normC{\tilde{\varphi}_{n}}^2 \leq 0,
\end{equation}
which is satisfied for $\normC{\tilde{\varphi}_n}\leq\mathcal{E}(\eta)$.
The conclusion is finally drawn thanks to Lemma~\ref{lem:convf}, which states that $\normC{\tilde{\varphi}_{n}}\leq \mathcal{E}(\eta)$ holds for any order $n$ greater than $\mathcal{N}(\mathcal{E}(\eta))$. 
\end{proof}

\begin{remark}
Notice that the maximal values $\kappa_1$ and $\kappa_2$ can easily be computed by grid search with an equally-spaced grid.\\
\end{remark}

Contrary to previous works based on discretization procedures with exponential kernels~\cite{Mondie2017,Mondie2021} or splines~\cite{Zhabko2019,Gu2013} which were limited to algebraic convergence rates, we take the benefits of Legendre polynomial approximation to obtain a supergeometric convergence rate on the remainder $|\tilde{V}_n(\varphi)|$. Therefore, the proposed convergence property of the remainder will be the key to design a stability test for time-delay systems that extends and enhances existing results.

\subsection{Necessary and sufficient stability test}

Lemmas~\ref{lem:CN} and~\ref{lem:CS} applied to the approximated Lyapunov-Krasovskii functional~$V_n$ defined by~\eqref{eq:Vn} provide a new necessary and sufficient condition of stability for system~\eqref{eq:tds}.\\

\begin{theorem}\label{thm:CNS}
System~\eqref{eq:tds} is exponentially stable if and only if matrix $\mathbf{P}_{\mathcal{N}(\mathcal{E}(\eta_0))}$ in~\eqref{eq:Pn} is positive definite where $\mathcal{N},\mathcal{E}$ are defined in~\eqref{eq:neps},~\eqref{eq:epsast}, respectively. 
\end{theorem}
\begin{proof}
Assume that system~\eqref{eq:tds} is exponentially stable. For any vector $\begin{bsmallmatrix}x\\\Phi_n\end{bsmallmatrix}$ in $\mathbb{R}^{(n+1)m}$ and $n$ in $\mathbb{N}$, define function $\varphi$ as 
\begin{equation}
    \varphi(\tau) = \left\{ \begin{array}{ll}\ell_n^\top(\tau)\Phi_n,\quad &\forall \tau\in[-h,0), \\x,\quad& \text{if } \tau=0.\end{array}\right.
\end{equation}
Applying Lemma~\ref{lem:CN} with $\varphi$ given above, there exists $\eta>0$ such that
\begin{equation*}
\begin{aligned}
    V(\varphi) = V_n(\varphi) = \begin{bsmallmatrix}x\\\Phi_n\end{bsmallmatrix}^{\top}\mathbf{P}_n\begin{bsmallmatrix}x\\\Phi_n\end{bsmallmatrix} 
   \geq \eta \norm{\begin{bsmallmatrix}x\\\Phi_n\end{bsmallmatrix}}^2,
\end{aligned}
\end{equation*}
which yields $\mathbf{P}_n\succ 0$, for all $n\in\mathbb{N}$, since $x$ and $\Phi_n$ are any independent vectors.\\
Concerning the sufficiency, assume by contradiction that system~\eqref{eq:tds} is not exponentially stable, and that $\mathbf{P}_{\mathcal{N}(\mathcal{E}(\eta_0))}\succ 0$. This means that there exists a characteristic root of~\eqref{eq:tds} with a positive real part. Consequently, Lemma~\ref{lem:CS} ensures that there necessarily exists $\varphi$ in $\mathcal{S}$ such that $V(\varphi)<-\eta_0$ and
\begin{equation}
    V_n(\varphi) = V(\varphi) - \tilde{V}_n(\varphi) \leq -\eta_0 + \norm{\tilde{V}_n(\varphi)},
\end{equation}
with $\eta_0$ given by~\eqref{eq:CS}. Finally, the convergence presented in Lemma~\ref{lem:conv} with $\eta=\eta_0$ leads to 
\begin{equation}
    V_n(\varphi) = \begin{bsmallmatrix}x\\\Phi_n\end{bsmallmatrix}^{\top}\mathbf{P}_n\begin{bsmallmatrix}x\\\Phi_n\end{bsmallmatrix}  \leq 0, \; \forall n\geq \mathcal{N}(\mathcal{E}(\eta_0)),
\end{equation}
which contradicts $\mathbf{P}_{\mathcal{N}(\mathcal{E}(\eta_0))}\succ 0$.
\end{proof}

The proposed theorem provides a numerical test to guarantee stability or instability of time-delay systems, which follows the following sequence.
\begin{enumerate}
    \item Compute $n^{\ast}=\mathcal{N}(\mathcal{E}(\eta_0))$ with $\eta_0$ given by~\eqref{eq:CS}.
    \item Evaluate each element of matrix $\mathbf{P}_{n^\ast}$.
    \item Test the positivity of matrix $\mathbf{P}_{n^{\ast}}$ to state the stability.% properties.
\end{enumerate}
Notice that this necessary and sufficient stability condition is formulated as in~\cite{Mondie2021} on the positivity of matrix $\mathbf{P}_n$ for a given order $n$.\\

As a background result, a hierarchical sufficient condition for instability of system~\eqref{eq:tds} is also formulated below.

\begin{corol}\label{cor:CNS}
If there exists $n\in\mathbb{N}$ such that matrix $\mathbf{P}_n$ given by~\eqref{eq:Pn} is not definite positive then system~\eqref{eq:tds} is not exponentially stable. Moreover, if this statement holds at an order $n$, then it also holds at the order $n+1$.
\end{corol}

\begin{proof}
Relying on the necessity part of the proof of Theorem~\ref{thm:CNS}, the sufficient condition for instability is trivial. The hierarchy can then be proven because matrix~$\mathbf{P}_{n+1}$ at order $n+1$ can be written as
\begin{equation*}
    \mathbf{P}_{\!n+1} \!=\! \begin{bsmallmatrix} \!\mathbf{P}_n & \frac{1}{h}\!\int_{-h}^0\! U^\top(h+\tau)A_d l_n(\tau)\dtau\\ \ast & \!\frac{1}{h^2}\!\int_{-h}^0\!\int_{-h}^0\!\! l_n(\tau_1)A_d^\top U(\tau_1-\tau_2)A_d l_n(\tau_2)\dtau_1\dtau_2 + \frac{h}{2n+1}I_m\!  \end{bsmallmatrix}.
\end{equation*}
If $\mathbf{P}_{n}$ is not positive definite then $\mathbf{P}_{n+1}$ cannot be positive definite.
\end{proof}

Interestingly, Corollary~\ref{cor:CNS} suggests an algortihm to solve the instability test. It consists in testing $\mathbf{P}_n\succ 0$ from $n=1$ to $n=n^\ast$. If $\mathbf{P}_n$ is not definite positive, then the system is unstable. Once order $n^\ast=\mathcal{N}(\mathcal{E}(\eta_0))$ is reached, the system is necessarily stable. 

It remains to solve the important problem of the numerical computation of matrix $\mathbf{P}_n$, which is necessarily to implement the algorithm. This is detailed in the next section.
%To implement numerically these new conditions, technical aspects on the computation of the elements of matrix $\mathbf{P}_n$ need to be detailed in the next section.

\section{Computational issues}

\subsection{Numerical issues}

To perform the numerical test presented above, each coefficient of matrix~$\mathbf{P}_n$ given by~\eqref{eq:Pn} needs to be evaluated numerically. It is worth noticing that this problem is not encountered in~\cite{Mondie2021} since the matrix to be evaluated contains point-wise evaluations of the Lyapunov matrix $U$. Here, the situation is more complicated since the matrix $\mathbf{P}_n$ given by~\eqref{eq:Pn} is equal to
\begin{equation}
    \mathbf{P}_n \!\!=\!\! \begin{bsmallmatrix} \!\!U(0) & Q_0^\top A_d & \cdots & Q_{n-1}^\top A_d\\ 
    \ast &\!\! A_d^\top (T_{00}+T_{00}^{\flat\top}) A_d + hI_m \!& \cdots & A_d^\top (T_{0n-1}+T_{0n-1}^{\flat\top}) A_d  \\
    \ast & \ast & \ddots & \vdots\\
    \ast & \ast & \ast & \!A_d^\top (T_{\underset{n-1}{n-1}}+T_{\underset{n-1}{n-1}}^{\flat\top}) A_d  + \frac{h}{2n-1}I_m\!
    \end{bsmallmatrix}\!,
\end{equation} 
where $Q_k,T_{jk}$ and $T_{jk}^\flat$ are the Legendre coefficients of the Lyapunov matrix $U$ given in the vector form as follows
\begin{equation}\label{eq:QTkvec}
\left\{
\begin{aligned}
    \mathcal{Q}_k = \mathrm{vec}(Q_k) &= \begin{bsmallmatrix}I_{m^2}&0\end{bsmallmatrix}\Gamma_k\N^{-1}\begin{bsmallmatrix}-\mathrm{vec}(I_m)\\0\end{bsmallmatrix},\\
    \mathcal{T}_{jk} = \mathrm{vec}(T_{jk}) &= \begin{bsmallmatrix}I_{m^2}&0\end{bsmallmatrix}\bar{\Gamma}_{jk}\N^{-1}\begin{bsmallmatrix}-\mathrm{vec}(I_m)\\0\end{bsmallmatrix},\\
    \mathcal{T}_{jk}^\flat = \mathrm{vec}(T_{jk}^\flat) &= \begin{bsmallmatrix}I_{m^2}&0\end{bsmallmatrix}\bar{\Gamma}_{jk}^\flat\N^{-1}\begin{bsmallmatrix}-\mathrm{vec}(I_m)\\0\end{bsmallmatrix},
\end{aligned}
\right.
\end{equation}
and where $\Gamma_k$, $\bar{\Gamma}_{jk}$ and $\bar{\Gamma}_{jk}^\flat$ are defined by
\begin{align}
    \Gamma_k &= \int_{-h}^0 \e^{(h+\tau)\M}l_k(\tau)\dtau,\; \forall k\in\{0,\dots,n-1\}, \label{eq:gammak}\\
    \bar{\Gamma}_{jk} &= \int_{-h}^0 \left(\int_{-h}^{\tau_1} \e^{(\tau_1-\tau_2)\M}l_k(\tau_2)\dtau_2\right) l_j(\tau_1) \dtau_1,\label{eq:gammajk1}\\
    \bar{\Gamma}_{jk}^\flat &= \int_{-h}^0 \left(\int_{\tau_1}^0 \e^{(\tau_2-\tau_1)\M}l_k(\tau_2)\dtau_2\right) l_j(\tau_1) \dtau_1,\label{eq:gammajk2}\\
    &\qquad\qquad\qquad\qquad\qquad \forall (j,k)\in\{0,\dots,n-1\}^2.\notag
\end{align}
The question of the numerical implementation of these integral terms in a reasonable time is then raised. Such computations can be done analytically by computer algebra systems but may turn out to be a tough task, especially for large $n$ or $m$. For instance, for $m=4$ and $n=100$, the exact calculation of~$\mathbf{P}_n$ can take days on a basic computer. An alternative computation through inductive relations is proposed to face the problem and make our results tractable numerically.

\subsection{Iterative calculation of Legendre exponential coefficients}

The numerical issue resumes to the calculation of Legendre polynomials coefficients of exponential matrices~\eqref{eq:gammak}-\eqref{eq:gammajk2}. To perform this computation recursively, the following relations can be used.

\begin{prop}\label{prop1}
If $\M$ is a non singular matrix, then matrices $\Gamma_k$ in~\eqref{eq:gammak} can be computed by the recursive relation 
\begin{equation}\label{eq:rec1}
    \Gamma_k = \Gamma_{k-2} - \frac{2(2k-1)}{h}\M^{-1}\Gamma_{k-1}, \quad\forall k\geq 2,
\end{equation} 
initialized with
\begin{equation}
    \begin{bsmallmatrix} \Gamma_0 \\ \Gamma_1 \end{bsmallmatrix} = \begin{bsmallmatrix} \M^{-1}(\e^{h\M}-I_{2m^2}) \\ \M^{-1}(\e^{h\M}+I_{2m^2}) - \frac{2}{h}\M^{-1}\Gamma_0 \end{bsmallmatrix}.
\end{equation}
\end{prop}

\begin{prop}\label{prop2}
For any matrix $\M$ and for matrices $\bar{\Gamma}_{jk}$, $\bar{\Gamma}^\flat_{jk}$ expressed in~\eqref{eq:gammajk1},\eqref{eq:gammajk2}, the following equality holds
\begin{equation}
    \bar{\Gamma}_{jk} = (-1)^{j+k}\bar{\Gamma}^\flat_{jk},\quad \forall (j,k)\in\{0,\dots,n-1\}^2.
\end{equation}
\end{prop}

\begin{prop}\label{prop3}
If $\M$ is a non singular matrix, then matrices $\bar{\Gamma}_{jk}$ in~\eqref{eq:gammajk1} can be computed by the following relations
\begin{equation}\label{eq:rec2}
\bar{\Gamma}_{jk} \!=\!
\left\{
    \begin{array}{ll}
        \!\!(-1)^{j+k}\bar{\Gamma}_{kj}, &\forall k\!<\!j,\\
        \!\!\!\begin{pmatrix}
        \!\bar{\Gamma}_{jk-2} \!+\!  \M^{-1}\! \frac{2(2k-1)}{h} \bar{\Gamma}_{jk-1}\! \\
        \!-\! \M^{-1}\! \frac{h}{2j+1}(\delta_{jk}-\delta_{jk-2})\end{pmatrix}\!, &\forall k\!\geq\! \mathrm{max}(2,j),
    \end{array}
\right.
\end{equation} 
initialized with
\begin{equation}
    \begin{bsmallmatrix} \bar{\Gamma}_{00}\\ \bar{\Gamma}_{01}\\ \bar{\Gamma}_{11}\end{bsmallmatrix}
    = \begin{bsmallmatrix}\M^{-1}(\Gamma_0-hI_{2m^2}) \\ -\M^{-1}\Gamma_1 \\ \M^{-1} \left((\frac{2}{h}\M^{-1}-I_{2m^2})\Gamma_1 - \frac{h}{3}I_{2m^2}\right)\end{bsmallmatrix}.
\end{equation}
\end{prop}
\begin{proof}
The proofs of these propositions are respectively postponed to Appendices~\ref{app1}, \ref{app2} and \ref{app3}.
\end{proof}
\begin{remark}
Notice that the case where $M$ is singular has not been discussed in this paper. It could for instance be treated using the Jordan canonical form.\\
\end{remark}

These propositions allows to propose a numerical solution to compute the integral terms~\eqref{eq:gammak}-\eqref{eq:gammajk2} by induction. The induction requires a number of operations in the range of $n^2$ and $m^2$. The necessary and sufficient condition of stability for time-delay systems presented in this paper becomes numerically tractable in a reasonable time. In the last section, our numerical test for stability is performed on two examples.

\section{Application to numerical examples}

\subsection{Presentation of the examples}

\begin{example}\label{ex1}
Consider~\eqref{eq:tds} with $A=1$ and $A_d=-2$.
\end{example}

\begin{example}\label{ex2}
Consider~\eqref{eq:tds} with $A=\begin{bsmallmatrix}0&0&1&0\\0&0&0&1\\-10-K&10&0&0\\5&-15&0&-0.25\end{bsmallmatrix}$ and $A_d=\begin{bsmallmatrix}0&0&0&0\\0&0&0&0\\K&0&0&0\\0&0&0&0\end{bsmallmatrix}$, for any $K>0$.
\end{example}

\subsection{Numerical results on stability analysis}

Theorem~\ref{thm:CNS} is applied for point-wise values of delays to evaluate the stability of these systems with respect to the delay. The results are gathered in Table~\ref{tab:ex1} and~\ref{tab:ex2} for Example~\ref{ex1} and~\ref{ex2}, respectively. For both systems, one can see that the maximal allowable delay $h$, which guarantees the stability, can be given with a precision $0.001$. 

\begin{table}[!t]
    \centering
    \caption{Evaluation and comparison of our necessary and sufficient test for stability for Example~\ref{ex1} with several delays.}
    \begin{tabular}{|c|c|c|c|c|}
    \hline
        Delay $h$ & Result & Order $n^\ast$ & CPU time & Order $n^\ast$~\cite{Mondie2021}  \\
        \hline
        $0.1$ & Stable & $4$ & $0.2$s & $36$ \\
        $0.604$ & Stable & $13$ & $0.9$s & $\simeq 10^8$\\
        $0.605$ & Unstable & $13$ & $0.9$s & $\simeq 10^8$ \\
        $2$ & Unstable & $24$ & $2.5$s & $\simeq 10^{12}$ \\
    \hline
    \end{tabular}
    \label{tab:ex1}
\end{table}
%($\e^{-32}$)

\begin{table}[!t]
    \centering
    \caption{Evaluation of our necessary and sufficient test for stability for Example~\ref{ex2} for $K=10$ with several delays $h$.}
    \begin{tabular}{|c|c|c|c|}
    \hline
        Parameters & Result & Order $n^\ast$ & CPU time  \\
        \hline
        $K=10$, $h=0.552$ & Stable & $65$ & $150$s \\
        $K=10$, $h=0.553$ & Unstable & $65$ & $150$s \\
    \hline
    \end{tabular}
    \label{tab:ex2}
\end{table}
%($\e^{-300}$)

The estimated order $n^\ast=\mathcal{N}(\mathcal{E}(\eta_0))$ for our necessary and sufficient test of stability is reported on Tables~\ref{tab:ex1} and~\ref{tab:ex2}. For Example~\ref{ex2}, a map of orders in the $(K,h)$ plan is depicted in Fig~\ref{fig:U4}. One can see that order $n^\ast$ increases as parameters $h$ and $K$ increase. The exact limit between stable (on the right) and unstable (on the left) regions obtained using D-partition is superposed. As emphasized in Remark~\ref{rem:Lyapcond}, this line is excluded to our criterion.
%Surprisingly, order $n^\ast$ do not depends on the location of the stability areas but rather on parameters $h$ and $K$. That means that the convergence towards the expected sets of stability is not asymptotic but reached at a finite order $n^\ast$. 

The CPU time spent to compute $n^\ast$, evaluate the components of matrix $\mathbf{P}_{n^\ast}$ by induction and test the positivity of $\mathbf{P}_{n^\ast}$ is also reported on both Tables. The computational load is obviously increasing as the order increases. Actually, the processing time is not only impacted by the positivity test but also by the calculation of the integrals terms in matrix $\mathbf{P}_n$. By using the iterative relation given by Propositions~\ref{prop1},~\ref{prop2} and~\ref{prop3}, this time grows with the square of the order $n^\ast$. %Nonetheless, since errors are accumulated in the recurrence relations, the Matlab precision requirement has to be refined as $n$ increases ($10^{-32}$ for $n\simeq 20$ against $10^{-300}$ for $n\simeq 70$) and explain the large difference between CPU time in Table~\ref{tab:ex1} and~\ref{tab:ex2}.
%It is also important to note that the time processing depends on the precision requested for the calculation of the coefficients $\Gamma_k,\bar{\Gamma}_{kl}$. Indeed, referring to Fig.~\ref{fig:U4}, we notice that the coefficient precision increases as much as Matlab precision increases.
%corresponds to the calculation of parameters $\Gamma_k,\bar{\Gamma}_{kl}$ with a Matlab precision small enough in order to reach a precision on these coefficients (ie. $10^{-32}$ for Table~\ref{tab:ex1} and $10^{-300}$ for Table~\ref{tab:ex1}) 

Other positivity tests of finite size $n^\ast$ for stability of time-delay systems have also been developed by Gomez et al. and Zhabko et al.. A brief comparison with~\cite{Mondie2021} is presented in Table~\ref{tab:ex1}. The approximation of the complete Lyapunov-Krasovskii functional is realized by Legendre polynomial approximation here, whereas a discretization with exponential kernels is used in~\cite{Mondie2021}. From the supergeometric convergence rate of Legendre approximation, our estimation of the minimal order $n^\ast=\mathcal{N}(\eps)$ satisfies $\eps\sim\e^{-n^\ast\log(n^\ast)}$. In~\cite{Mondie2021}, this estimation is limited to $\eps\sim\frac{1}{n^\ast}$. On Table~\ref{tab:ex1}, this order of magnitude significant difference is highlighted. Notice that this advantage has to be balanced with the shape of the matrix $\mathbf{P}_n$ and the ease of computing its components.

Finally, recalling the sufficient criterion of instability mentioned in Corollary~\ref{cor:CNS}. It permits to detect some unstable systems by testing $\mathbf{P}_n\succ 0$ for a limited number of order. For Example~\ref{ex2}, for low orders $n=\{1,\dots,5\}$, the corresponding unstable areas with respect to parameters $(K,h)$, denoted $\{\mathcal{U}_1,\dots,\mathcal{U}_5\}$, are drawn on Fig.~\ref{fig:U5}. These tests are not time consuming and already get an accurate information on the unstable regions of Example~\ref{ex2}. Notice that $\mathcal{U}_1$ already spans the main areas of instability. The hierarchical structure $\mathcal{U}_1\subset\dots\subset\mathcal{U}_5$ is also verified. More interestingly, the hard-to-reach areas are located when the eigenvalues crosses the imaginary axis from the left-half plane to the right-half plane (see red lines).

\begin{figure}[!t]
    \centering
    \includegraphics[width=9cm]{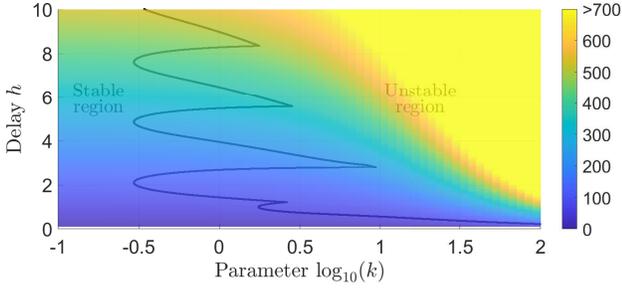}
    \caption{Example~\ref{ex2}: Required orders with respect to $(K,h)$.}
    \label{fig:U4}
\end{figure}

\begin{figure}[!t]
    \centering
    \includegraphics[width=9cm]{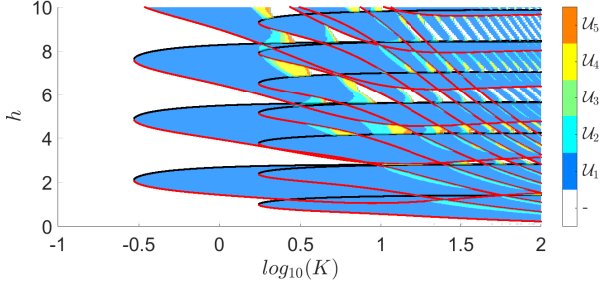}
    \caption{Example~\ref{ex2}: Unstable areas with respect to $(K,h)$.}
    \label{fig:U5}
\end{figure}

%\begin{figure}[!t]
%    \centering
%    \includegraphics[width=8cm]{figure_kh2.eps}
%    \caption{Example~\ref{ex2}: Stability with respect to $(K,h)$.}
%    \label{fig:U4}
%\end{figure}

\section{Conclusions and perspectives}
%This contribution is an extension to the necessary and sufficient condition for stability of time-delay systems developed in~\cite{Mondie2021}.
This paper has been devoted to the formulation of a necessary and sufficient stability condition of time-delays systems, extending the result of paper~\cite{Mondie2021}. It derives from the positivity of the complete Lyapunov-Krasovskii functional, where an approximation of the Lyapunov matrix has been considered. Stability can then be linked with the positive definiteness of a certain matrix of finite size $n^\ast$, which depends on systems parameters. The originality of our work relies on the approximation techniques which employs polynomial coefficients instead of discretized elements. The supergeometric convergence rate satisfied by Legendre approximation demonstrates that our condition requires smaller orders $n^\ast$ than in~\cite{Mondie2021}. Based on recurrence relations satisfied by Legendre polynomials, our stability criterion can be easily implemented by induction.

Our approach still to be generalized to distributed~\cite{Mondie2017} and neutral~\cite{Mondie2016} time-delay systems. New tracks of research would be to cover other classes of delay systems and more complex infinite-dimensional systems by this methodology. The Legendre polynomial approximation technique could also be deployed to problems of controller or observer synthesis for infinite-dimensional systems.

\appendix

\section{Proofs of the necessary and sufficient stability conditions}

\subsection{Proof of Lemma~\ref{lem:CN}}\label{app:CN}

\begin{proof}
The proof can be found in \cite[Th.~3]{Mondie2017} with $W=I_m$ and follows arguments given in \cite[Th.~5.19]{Gu2003}. Firstly, we introduce the functional
\begin{equation}
    W(\varphi) = V(\varphi) - \int_{-h}^0 \begin{bsmallmatrix}\varphi(0)\\\varphi(\tau)\end{bsmallmatrix}^\top \begin{bsmallmatrix} \frac{\eta_0}{h} I_m & 0\\ 0 & \frac{1}{2}I_m\end{bsmallmatrix} \begin{bsmallmatrix}\varphi(0)\\\varphi(\tau)\end{bsmallmatrix}\dtau.
\end{equation}
Functional~$V$ is built so that the time derivative along the trajectories $x_t$ of system~\eqref{eq:tds} gives~\eqref{eq:Vdot}. Then, the time derivative of the functional $W$ gives
\begin{equation}
    \dot{W}(x_t) = - \begin{bsmallmatrix}x(t)\\x(t-h)\end{bsmallmatrix}^\top \begin{bsmallmatrix} \eta_0 \mathcal{H}(A) + \frac{1}{2}I_m &  \eta_0 A_d \\ \ast &  \frac{1}{2}I_m\end{bsmallmatrix} \begin{bsmallmatrix}x(t)\\x(t-h)\end{bsmallmatrix},
\end{equation}
for which there exists a sufficiently small $\eta_0>0$ such that $\dot{W}(x_t)\leq 0$. Then, integrating from $0$ to $\infty$ and assuming the exponential stability of~\eqref{eq:tds} yields $W(x_0) \geq 0$, for any initial conditions $x_0=\varphi$ in $\mathcal{C}_{pw}(-h,0;\mathbb{R}^{m})$. Thus, \eqref{eq:CN} holds with $\eta=\mathrm{min}(\eta_0,\frac{h}{2})$.
\end{proof}

\subsection{Proof of Lemma~\ref{lem:CS}}\label{app:CS}

\begin{proof}
The following proof is close to~\cite[Appendix]{Mondie2021} and to~\cite[Lemma 5]{Zhabko2019}. The only difference relies on the definition of set $\mathcal{S}$ that has been extended to $\mathcal{C}_{\infty}$ instead of $\mathcal{C}_{1}$. Let us denote $s_0=\alpha+ i\beta$ be the eigenvalue with positive real part of system~\eqref{eq:tds}. According to~\cite{Mondie2017lemma}, there exists a vector $C=C_1+iC_2$ such that $\norm{C_2}\leq\norm{C_1}=1$, $C_2^\top C_1=0$ and that $$(s_0I_m-A-A_d\e^{-hs_0})C=0.$$ Consequently, we have $ \alpha \leq \norm{s_0} \leq \norm{A} + \norm{A_d} = r$,
%\begin{equation}\label{eq:s0}
%    \alpha \leq \norm{s_0} \leq \norm{A} + \norm{A_d} = r,
%\end{equation}
and
\begin{equation*}\label{eq:xtsol}
    \bar{x}(t) = \e^{\alpha t} \left(\mathrm{cos}(\beta t)C_1-\mathrm{sin}(\beta t)C_2\right),\;\forall t\in\mathbb{R},
\end{equation*}
is a solution to system~\eqref{eq:tds}. Then, thanks to~\eqref{eq:Vdot}, the derivatives of~$V(\bar{x}_t)$ along the trajectories of~\eqref{eq:tds} yields to 
\begin{equation}\label{eq:Vd}
    \dot{V}(\bar{x}_t) = - \norm{\bar{x}(t-h)}^2.
\end{equation}
After some calculations developed in~\cite{Mondie2021}, we also obtain that
\begin{equation}
    V(\bar{x}_0) \leq - \frac{\e^{-2\alpha h}}{4\alpha}\mathrm{cos}^2(b_0) \leq - \frac{\e^{-2rh}}{4r}\mathrm{cos}^2(b_0) = -\eta_0.
\end{equation}
Let $\varphi=\bar{x}_0$, which belongs to $\mathcal{C}_{\infty}(-h,0;\mathbb{R}^m)$ and satisfies $\norm{\varphi(0)}=\norm{x(0)}=1$. We finally prove by induction that $\norm{\varphi^{(k)}(\tau)}\leq r^k \e^{\alpha \tau}$, for any $\tau$ in $[-h,0]$. Initially, $\norm{\bar{x}(\tau)}\leq \e^{\alpha \tau}$ holds on $[-h,0]$. Then, assuming $\norm{\bar{x}^{(k)}(\tau)}\leq r^k \e^{\alpha \tau}$, since $\bar{x}$ satisfies~\eqref{eq:tds} and is infinitely differentiable, we obtain
\begin{equation*}
    \norm{\bar{x}^{(k+1)}(\tau)} \leq \norm{A}\norm{\bar{x}^{(k)}(\tau)} + \norm{A_d}\norm{\bar{x}^{(k)}(\tau-h)} \leq r^{k+1} \e^{\alpha \tau}.
\end{equation*}
Therefore, $\normC{\varphi^{(k)}}=\underset{[-h,0]}{\mathrm{sup}}\norm{\varphi^{(k)}(\tau)}\leq r^k$ for any $k\in\mathbb{N}$.
\end{proof}

\section{Proofs of the recursive relations}

\subsection{Proof of Proposition~\ref{prop1}}\label{app1}

\begin{proof}
The proof is based on the relation 
\begin{equation}\label{eq:legrec}
    l'_{k} - l'_{k-2} = \frac{2(2k-1)}{h}l_{k-1},\quad \forall k\geq 2,
\end{equation}
satisfied by Legendre polynomials~\cite{Gautschi2006legendre}.
To compute $\Gamma_k$, an integration by parts leads to
\begin{equation*}
\begin{array}{rcl}
    \Gamma_{k} - \Gamma_{k-2} &=&
    \displaystyle \int_{-h}^0\!\! \e^{(h+\tau)\M}(l_{k}\!-\!l_{k-2})(\tau)\dtau,\\
    &=& - \frac{2(2k-1)}{h}\M^{-1}\! \displaystyle\int_{-h}^0 \!\! e^{(h+\tau)\M}l_{k-1}(\tau)\dtau \\
    &&+\M^{-1}\left[\e^{(h+\tau)\M}(l_{k}-l_{k-2})(\tau)\right]_{-h}^0,
\end{array}
\end{equation*}
and knowing that $l_{k}(-h)=l_{k-2}(-h)=(-1)^{k}$ and $l_{k}(0)=l_{k-2}(0)=1$, the last term vanishes. Moreover, for $k\in\{0,1\}$, we directly obtain
\begin{equation*}
\begin{aligned}
    \Gamma_0 \!\!&=\!\! \int_{-h}^0\!\!\e^{(h+\tau)\M}\dtau \!=\! \M^{-1}(\e^{h\M}-I_{2m^2}),\\
    \Gamma_1 \!\!&=\!\! \int_{-h}^0\!\!\e^{(h+\tau)\M}\textstyle\left(\frac{2\tau+h}{h}\right)\dtau
    \!=\! \M^{-1}(\e^{h\M}\!+I_{2m^2}\!) \!-\! \frac{2}{h}\M^{-1}\Gamma_0,
\end{aligned}
\end{equation*}
which concludes the proof.
\end{proof}

\subsection{Proof of Proposition~\ref{prop2}}\label{app2}

\begin{proof}
The successive changes of variables $\tau_2'=-(\tau_2+h)$ and $\tau_1'=-(\tau_1+h)$ directly lead to
\begin{equation*}
\begin{aligned}
    \Gamma_{jk} &=\!\! \int_{-h}^0\!\! \left(\!\int_{0}^{-(\tau_1+h)}\!\!\!\!\!\! \e^{(\tau_1+\tau_2'+h)\M}l_k(-\tau_2'\!-\!h)\dtau_2\!\right)\! l_j(\tau_1) \dtau_1,\\
    &= \!\!\int_{0}^{-h}\!\!\left(\!\int_{0}^{\tau_1'}\!\!\e^{(\tau_2'-\tau_1')\M}l_k(-\tau_2'\!-\!h)\dtau_2'\!\right)\!l_j(-\tau_1'\!-\!h)\dtau_1',\\
    &= (-1)^{j+k}\Gamma_{jk}^\flat,
\end{aligned}
\end{equation*}
following the parity properties of Legendre polynomials, i.e. $l_k(-\tau-h)=(-1)^kl_k(\tau)$ for all $\tau$ in $[-h,0]$.
\end{proof}

\vspace{-0.4cm}

\subsection{Proof of Proposition~\ref{prop3}}\label{app3}

\begin{proof}
As in Appendix~\ref{app1}, an integration by parts and~\eqref{eq:legrec} ensure that $\bar{\Gamma}_{jk}^+$ satisfies the recursive relation
\begin{equation*}
\begin{array}{l}
    \bar{\Gamma}_{jk} - \bar{\Gamma}_{jk-2} \\
    =
    \displaystyle \int_{-h}^0\!\!\left(\int_{-h}^{\tau_1}\!\! \e^{(\tau_1-\tau_2)\M}(l_{k}\!-\!l_{k-2})(\tau_2)\dtau_2\right)l_j(\tau_1)\dtau_1,\\
    = \! \frac{2(2k-1)}{h}\M^{-1}\!\!\!\displaystyle \int_{-h}^0 \!\!\!\left(\!\int_{-h}^{\tau_1}\!\! \e^{(\tau_1-\tau_2)\M}l_{k-1}(\tau_2)\dtau_2\!\right)\!l_j(\tau_1)\dtau_1\\
    \qquad \displaystyle - \M^{-1} \int_{-h}^0 \!\!(l_{k}-l_{k-2})(\tau_1)l_j(\tau_1)\dtau_1,\\
    = \frac{2(2k-1)}{h}\M^{-1}\bar{\Gamma}_{jk-1} - \frac{h}{2j+1}\M^{-1}\big(\delta_{jk}-\delta_{jk-2}\big).
\end{array}
\end{equation*}
It is also important to notice that  
\begin{equation*}
\begin{aligned}
    \mathcal{D} &= \{(\tau_1,\tau_2);\,\tau_1\in[-h,0],\tau_2\in[-h,\tau_1]\},\\
    &= \{(\tau_1,\tau_2);\,\tau_2\in[-h,0],\tau_1\in[\tau_2,0]\},
\end{aligned}
\end{equation*}
which means that, using the change of coordinates $\tau_1'=\tau_2$ and $\tau_2'=\tau_1$, the following equations hold
\begin{equation*}
\begin{aligned}
    \bar{\Gamma}_{jk} &= \int_0^{-h}\!\!\int_{-h}^{\tau_1}\!\!\e^{(\tau_1-\tau_2)\M}l_j(\tau_1)l_k(\tau_2)\dtau_2\dtau_1,\\
    &= \int_0^{-h}\!\!\int_{\tau_2}^{0}\!\!\e^{(\tau_1-\tau_2)\M}l_j(\tau_1)l_k(\tau_2)\dtau_1\dtau_2,\\
    &= \int_0^{-h}\!\!\int_{\tau_1'}^{0}\!\!\e^{(\tau_2'-\tau_1')\M}l_k(\tau_1')l_j(\tau_2')\dtau_2'\dtau_1',\\
    &= \bar{\Gamma}_{kj}^\flat = (-1)^{j+k}\bar{\Gamma}_{kj}.
\end{aligned}
\end{equation*}
An integration by parts yields the initial values.
%The initial values are given by
%\begin{equation*}
%\begin{array}{rcl}%\bar{\Gamma}_{00}^\flat =
%    \bar{\Gamma}_{00} &=& \displaystyle \int_{-h}^0\!\!\left(\int_{-h}^{\tau_1}\!\! \e^{(\tau_1-\tau_2)\M}\dtau_2\right)\dtau_1,\\
%    &=& \displaystyle \int_{-h}^0\!\!\left(-\M^{-1}(I_{2m^2}-\e^{(\tau_1+h)\M})\right)\dtau_1,\\
%    &=& -h\M^{-1} + \M^{-1}\Gamma_0,
%    %&=& -h\M^{-1} + \M^{-2}(\e^{h\M}-I_{2m^2}).
%\end{array}
%\end{equation*}
%then by 
%\begin{equation*}
%\begin{array}{rcl}
%    \bar{\Gamma}_{01} &=&
%    -\bar{\Gamma}_{10},\\
%    &=& -\displaystyle \int_{-h}^0\!\!\left(\int_{-h}^{\tau_1}\!\! \e^{(\tau_1-\tau_2)\M}\dtau_2\right)\!\!\left(\frac{2\tau_1+h}{h}\right)\dtau_1,\\
%    &=& \displaystyle \M^{-1}\!\! \!\int_{\!-h}^0\!\!\!\left(\!I_{2m^2}\!-\!\e^{(\tau_1+h)\M}\!\right)\!\!\left(\!\frac{2\tau_1\!+\!h}{h}\!\right)\!\dtau_1,\\
%    &=& -\M^{-1}\Gamma_1.
%\end{array}
%\end{equation*}
%and finally by
%\begin{equation*}
%\begin{array}{rcl}%\bar{\Gamma}_{11}^\flat = 
%    \bar{\Gamma}_{11} &=& \int_{-h}^0\!\!\left(\int_{-h}^{\tau_1}\!\! \e^{(\tau_1-\tau_2)\M}\!\left(\!\frac{2\tau_2+h}{h}\!\right)\!\dtau_2\!\right)\!\!\left(\frac{2\tau_+h}{h}\!\right)\!\dtau_1,\\
%    &=& \frac{2}{h}\M^{-1} \bar{\Gamma}_{10}\\
%    &-&\!\M^{-1}\!\int_{-h}^0\!\!\left[\e^{(\tau_1-\tau_2)\M}\!\left(\frac{2\tau_2+h}{h}\right)\!\right]_{-h}^{\tau_1}\!\!\left(\frac{2\tau_1+h}{h}\right)\!\dtau_1,\\
%    &=& \frac{2}{h}\M^{-2}\Gamma_1 - \frac{h}{3}\M^{-1} - \M^{-1}\Gamma_1,
%\end{array}
%\end{equation*}
%which concludes the proof.
\end{proof}
\vspace{-0.4cm}

\bibliographystyle{plain}        % Include this if you use bibtex 
\bibliography{ref} 

\end{document}